\documentclass[a4paper,10pt]{article}
\usepackage[utf8]{inputenc}
\usepackage{centernot}
\usepackage{url}
\usepackage[T1]{fontenc}
\usepackage{amsmath}
\usepackage{amsfonts}
\usepackage{amssymb}
\usepackage{amsthm}
\usepackage{mathtools}
\usepackage[unicode]{hyperref}
\usepackage{tikz}
\usepackage[top=2cm, bottom=2cm, left=2.5cm, right=2.5cm]{geometry}
\newtheorem{theorem}{Theorem}[section]
\newtheorem{lemma}[theorem]{Lemma}

\theoremstyle{definition}

\theoremstyle{remark}
\newtheorem{rem}{Remark}[section]

\newcommand\Nbd{\operatorname{Nbd}}
\newcommand\Var{\operatorname{Var}}

\newcommand\Capa{\operatorname{Cap}}

\allowdisplaybreaks

\begin{document}
\hypersetup{pageanchor=false}
\pgfsetxvec{\pgfpoint{8pt}{0}}
\pgfsetyvec{\pgfpoint{0}{8pt}}
\author{Yinshan Chang\footnote{Department of Mathematics, University of Sichuan, Chengdu, China. Email: ychang@scu.edu.cn
Max-Planck-Institute for Mathematics in the Sciences, 04103, Leipzig, Germany.}}
\title{Two observations on the capacity of the range of simple random walks on $\mathbb{Z}^3$ and $\mathbb{Z}^4$}
\date{}
\maketitle

\begin{abstract}
We prove a weak law of large numbers for the capacity of the range of simple random walks on $\mathbb{Z}^{4}$. On $\mathbb{Z}^{3}$, we show that the capacity, properly scaled, converges in distribution towards the corresponding quantity for three dimensional Brownian motion.
\end{abstract}

\section{Introduction}
Let $(X_n)_n$ be a simple random walk on $\mathbb{Z}^d$ with $d=3$ or $4$. We are interested in the scaling limit of the capacity of the random set
\[X[0,n]\overset{\text{def}}{=}\{X_0,\ldots,X_n\},\]
where the capacity $\Capa(F)$ of a set $F$ of vertices on $\mathbb{Z}^d$ is defined as the sum of escaping probabilities:
\begin{equation}\label{defn: capacity srw}
 \Capa(F)\overset{\text{def}}{=}\sum_{x\in F}\mathbb{E}^{x}[\forall n\geq 1, X_n\notin F].
\end{equation}
The capacity of the range of random walks is closely related with the intersection probability of two independent random walks. In fact, many estimations on $\Capa(X[0,n])$ were deduced from that. We refer the reader to the Lawler’s classical book \cite{LawlerMR2985195} and the reference there.

In the present paper, we show that when $d=4$, the second moment of $\Capa(F)$ is asymptotically equivalent to the square of the first moment, which implies a weak law of large numbers:
\begin{theorem}\label{thm: concentration capacity 4d}
For a SRW $(X_n)_n$ on $\mathbb Z^4$,
\begin{equation}
 \lim_{n\rightarrow\infty}\Var(\Capa(X[0,n]))/\mathbb{E}[\Capa(X[0,n])]^2=0,
\end{equation}
which implies that
\begin{equation}\label{eq: 4d weak law of large numbers}
 \Capa(X[0,n])/\mathbb{E}[\Capa(X[0,n])]\overset{\text{Probability}}\rightarrow 1,\quad n\rightarrow\infty.
\end{equation}
\end{theorem}
Such a weak law of large numbers was conjectured by Asselah, Schapira and Sousi \cite[Section~6]{SRWCapCLT}. Besides, they also expect a random scaling limit for $\Capa(X[0,n])/\mathbb{E}[\Capa(X[0,n])]$ for $d=3$ and there is no such kind of weak law of large numbers on $\mathbb{Z}^3$. We affirm this as a corollary (see Remark~\ref{rem: 3d no weak LLN}) of our second main result, which states that as $n\rightarrow\infty$, $\Capa(X[0,n])/\sqrt{n}$ has a random limit in distribution, which is the corresponding quantity for three dimensional Brownian motion. To be more precise, let $(M_t)_{t\geq 0}$ be the standard Brownian motion on $\mathbb{Z}^{3}$. Recall the Green function for Brownian motions on $\mathbb{R}^{3}$, see e.g. \cite[Theorem~3.33]{MortersPeresMR2604525}:
 \[G(x,y)=\int_{0}^{\infty}(2\pi t)^{-3/2}e^{-||x-y||_2^2/2t}=\frac{1}{2\pi}||x-y||_2^{-1}.\]
The corresponding (Brownian motion) capacity of a Borel set $F$ is given by
 \begin{equation}\label{eq: bm green capacity}
  \Capa_{BM}(F)^{-1}=\inf\left\{\int\int G(x,y)\mu(\mathrm{d}x)\mu(\mathrm{d}y):\mu\text{ is a probability measure on }F\right\},
 \end{equation}
 see e.g. \cite[Definition~8.18]{MortersPeresMR2604525}. We have the following result on the fluctuation of $\Capa(X[0,n])$ on $\mathbb{Z}^3$.
\begin{theorem}\label{thm: scaling limit capacity 3d}
For a SRW $(X_n)_n$ on $\mathbb Z^3$, as $n\rightarrow\infty$, $\Capa(X[0,n])/\sqrt{n}$ converges to $\frac{1}{3\sqrt{3}}\Capa_{BM}(M[0,1])$ in distribution.
\end{theorem}

The law of large numbers for $(\Capa(X[0,n]))_{n}$ had already been obtained in dimension $5$ and larger by Jain and Orey \cite{JainOreyMR0243623}. In \cite{SRWCapCLT}, Asselah, Schapira and Sousi established a central limit theorem in dimension larger than or equal to $6$. To understand the model of random interlacement invented by Sznitman \cite{SznitmanMR2680403}, R\'{a}th and Sapozhnikov \cite{RathSapozhnikovMR2889752,RathSapozhnikovMR2819660} established moments and deviation bounds for the capacity of the union of ranges of paths. During the study of the simple random walk loop percolation on $\mathbb{Z}^d$ \cite{loop-perc-Zd}, together with Sapozhnikov, when $d=4$, we improved the upper bounds for the second moment of $\Capa(X[0,n])$ by showing that it is comparable with the square of the first moment. Theorem~\ref{thm: concentration capacity 4d} sharpens our result in \cite{loop-perc-Zd}, which implies a weak law of large numbers for $\Capa (X[0,n])_{n}$ on $\mathbb{Z}^4$. Soon after this and very recently, Asselah, Schapira and Sousi \cite{SRWCapZ4} greatly improved the result in $\mathbb{Z}^{4}$ by proving the strong law of large numbers and the central limit theorem. In another paper \cite{CapWienerSausageD4} by the same authors, the strong law of law numbers was established for the Wiener sausage, which is the continuous counterpart of the discrete simple random walk. We refer the reader to \cite{CapWienerSausageD4} for more references and historical remarks on the Wiener sausages.

Finally, we briefly outline the proof. The argument for Theorem~\ref{thm: concentration capacity 4d} is a refinement of that in \cite{loop-perc-Zd}. We consider two independent simple random walks $(X^{(0)}_i)_{i=0,\ldots,n}$ starting from $0$ and $(X^{(1)}_i)_{i\geq 0}$ starting from ``the infinity''. Equivalent to the estimation of $\Capa(X^{0}[0,n])$, we estimate the intersection probability of $(X^{(0)}_i)_{i=0,\ldots,n}$ and $(X^{(1)}_i)_{i\geq 0}$. We use the strong Markov property at time $\tau_1$, where $\tau_1$ is the first time that $(X^{(0)}_i)_{i=0,\ldots,n}$ meets the trajectory of $(X^{(1)}_i)_{i\geq 0}$. A key observation is that $\mathbb{P}[\tau_1/n\in\cdot|\tau_1\leq n]$ converges to the uniform distribution on $[0,1]$ when $n\to\infty$. Together with the sharp estimate
\[\mathbb{E}[\Capa(X[0,n])]\overset{n\to\infty}{\sim}\frac{\pi^2}{8}\frac{n}{\log n}\]
from \cite{SRWCapCLT}, we conclude the desired result. Theorem~\ref{thm: scaling limit capacity 3d} is proved via a coupling between SRW paths and Brownian motion paths. We crucially use the fact that two independent SRW (Brownian motion) paths  are very likely to intersect for $d=3$, by the result of Lawler \cite[Lemma~2.4,2.6]{LawlerMR1423466}.

\paragraph{Organization of the paper}
We introduce necessary notation in Section~\ref{sect: notation}. Then, we prove Theorem~\ref{thm: concentration capacity 4d} and \ref{thm: scaling limit capacity 3d} in separate sections.
\section{Notation}\label{sect: notation}
We collect several notation in the following.
\begin{itemize}
 \item $\ell_1$-balls on $\mathbb{Z}^d$: $B_{\mathbb{Z}^d}(x,r)=\{z\in\mathbb{Z}^d:|x-z|_1\leq r\}$ for $x\in\mathbb{Z}^d$ and $r\geq 0$.
 \item $\ell_1$-balls on $\mathbb{R}^d$: $B_{\mathbb{R}^d}(x,r)=\{z\in\mathbb{R}^d:|x-z|_1\leq r\}$ for $x\in\mathbb{R}^d$ and $r\geq 0$.
 \item Simple random walk: $(X_n)_{n\geq 0}$.
 \item Brownian motion: $(M_t)_{t\geq 0}$.
 \item Range of a SRW: $X[0,n]=\{X_0,\ldots,X_n\}$.
 \item Range of a Brownian motion: $M[0,t]=\cup_{s\in[0,t]}\{M_s\}$.
 \item First entrance time for a set $F$: $\tau(F)=\inf\{n\geq 0:X_n\in F\}$.
 \item Hitting time for a set $F$: $\tau^{+}(F)=\inf\{n\geq 1:X_n\in F\}$.
 \item Green function for SRWs: $G(x,y)=\mathbb{P}^{x}[\sum_{n\geq 0}1_{X_n=y}]$.
 \item Green function for Brownian motions on $\mathbb{R}^d$ ($d\geq 3$):
 \[G(x,y)=\int_{0}^{\infty}(2\pi t)^{-d/2}e^{-||x-y||_2^2/2t}=\frac{\Gamma(d/2)}{(d-2)\pi^{d/2}}||x-y||_2^{2-d}.\]
 \item SRW capacity of a set $F$: $\Capa(F)=\sum_{x\in F}\mathbb{E}^{x}[\tau^{+}(F)=\infty]$.
 \item Brownian motion capacity of a Borel set $F$:
 \begin{equation*}
  \Capa_{BM}(F)^{-1}=\inf\{\int\int G(x,y)\mu(\mathrm{d}x)\mu(\mathrm{d}y):\mu\text{ is a probability measure on }F\}.
 \end{equation*}
\end{itemize}

\section{\texorpdfstring{Four dimension: concentration of $\Capa(X[0,n])$ around its mean}{Four dimension: concentration of Cap(X[0,n]) around its mean}}
We prove Theorem~\ref{thm: concentration capacity 4d} in this section. Before that, we need to state three auxiliary lemmas.

\begin{lemma}[{\cite[Proposition 4.6.4]{LawlerMR2677157}}]\label{lem: hitting probability}
 For a transient graph and a subset $F$ of vertices, by last passage time decomposition,
\[\mathbb{P}^x[\tau(F)<\infty]=\sum\limits_{z\in F}G(x,z)\mathbb{P}^z[\tau^{+}(F)=\infty].\]
\end{lemma}

\begin{lemma}[{\cite[Theorem 4.3.1]{LawlerMR2677157}}]\label{lem: Green function}
 For a simple random walk on $\mathbb{Z}^d$, $d\geq 3$, there exist $0<c(d)\leq C(d)<\infty$ such that
 $$c(d)(1+||x-y||_{\infty})^{2-d}\leq G(x,y)\leq C(d)(1+||x-y||_{\infty})^{2-d}.$$
 More precisely, $G(x,y)=\frac{d\Gamma(d/2)}{(d-2)\pi^{d/2}}(||x-y||_{2}+1)^{2-d}+O((||x-y||_2+1)^{-d})$ as $||x-y||_2\to \infty$.
\end{lemma}

\begin{lemma}[{\cite[Corollary~1.4]{SRWCapCLT}}]\label{lem: mean capacity asymptotics 4d}
 For a SRW $(X_n)_n$ on $\mathbb Z^4$,
 \begin{equation}\label{eq: asymp capacity 4d}
  \lim_{n\rightarrow\infty}\frac{\log n}{n}\mathbb{E}[\Capa(X[0,n])]=\frac{\pi^2}{8}.
 \end{equation}
\end{lemma}

It is known that the capacity of a set is closely related to the hitting probability of that set, see Lemma~\ref{lem: hitting probability}. We will prove Theorem~\ref{thm: concentration capacity 4d} by refining the argument in \cite[Lemma~2.4]{loop-perc-Zd}.

\begin{proof}[Proof of Theorem~\ref{thm: concentration capacity 4d}]
Let $(X^0_n)_{n\geq 0},(X^1_n)_{n\geq 0},(X^2_n)_{n\geq 0}$ be three independent simple random walks. Denote by $\mathbb{E}_{(i)}^{x}$ the expectation corresponding to the random walk $X^i$ with initial point $x$. Similarly, we define $(\mathbb{E}_{(i),(j)}^{x,y})_{i\neq j}$. For simplicity of notation, we denote by $\mathbb{E}^{x,y,z}$ (or $\mathbb{P}^{x,y,z}$) the expectation (or probability) corresponding to $X^0,X^1$ and $X^2$ with initial points $x,y$ and $z$, respectively. Recall that $X^0[0,n]$ is the range of $X^0$ up to time $n$. Similarly, we define $X^1[0,\infty)$ and $X^2[0,\infty)$.

Let $x_0 = (Kn,0,0,0)$, where $K$ will be sent to infinity in the end. By Lemmas \ref{lem: hitting probability} and \ref{lem: Green function},
\begin{align*}
 \Capa(X^{0}[0,n])=\mathbb{P}_{(1)}^{x_0}[X^{0}[0,n]\cap X^{1}[0,\infty)\neq\emptyset]\cdot s(4)K^{2}n^{2}(1+O(K^{-1})),
\end{align*}
where $s(4)=\frac{(d-2)\pi^{d/2}}{d\Gamma(d/2)}\Big|_{d=4}=\pi^2/2$. Hence,
\begin{align*}
 \mathbb{E}[\Capa(X^{0}[0,n])]=&s(4)K^{2}n^{2}(1+O(K^{-1}))\\
 &\times \mathbb{P}^{0,x_0}\left[X^0[0,n]\cap X^1[0,\infty)\neq\emptyset\right],\\
 \mathbb{E}[(\Capa(X^{0}[0,n]))^2]=&s(4)^2K^{4}n^{4}(1+O(K^{-1}))\\
 &\times \mathbb{P}^{0,x_0,x_0}\left[X^0[0,n]\cap X^1[0,\infty)\neq\emptyset,~X^0[0,n]\cap X^2[0,\infty)\neq\emptyset\right].
\end{align*}
For $i=1$ and $2$, define $\tau_i=\inf\{j\geq 0:X^0_j\in X^{i}[0,\infty)\}$. By symmetry,
\begin{equation}
\mathbb{P}^{0,x_0,x_0}\left[X^0[0,n]\cap X^1[0,\infty)\neq\emptyset,~X^0[0,n]\cap X^2[0,\infty)\neq\emptyset\right]\leq 2\mathbb{P}^{0,x_0,x_0}[\tau_1\leq\tau_2\leq n].
\end{equation}
By conditioning on $X^1$ and $X^2$ and then applying the strong Markov property for $X^0$ at time $\tau_1$,
\begin{align*}
 \mathbb{P}^{0,x_0,x_0}[\tau_1\leq\tau_2\leq n]=&\mathbb{E}^{0,x_0,x_0}\left[\tau_1\leq\tau_2,\tau_1\leq n,\mathbb{E}_{(0)}^{X^0_{\tau_1}}\left[1_{\{X^0[0,n-\tau_1]\cap X^2[0,\infty)\neq\emptyset\}}\right]\right]\\
 \leq&\mathbb{E}^{0,x_0,x_0}\left[\tau_1\leq n,\mathbb{E}_{(0)}^{X^0_{\tau_1}}\left[1_{\{X^0[0,n-\tau_1]\cap X^2[0,\infty)\neq\emptyset\}}\right]\right].
\end{align*}
Then, we take the expectation with respect to $X^2$ and get that
\begin{equation*}
 \mathbb{P}^{0,x_0,x_0}[\tau_1\leq\tau_2\leq n]\leq \mathbb{E}_{(0),(1)}^{0,x_0}[X^0[0,n]\cap X^1[0,\infty)\neq\emptyset,\sup\limits_{y\in B(0,n)}\mathbb{P}_{(0),(2)}^{y,x_0}[X^0[0,n-\tau_1]\cap X^2[0,\infty)\neq\emptyset]].
\end{equation*}
By last passage time decomposition (Lemma~\ref{lem: hitting probability}) and the Green function estimate (Lemma~\ref{lem: Green function}) (or by gradient estimate for harmonic functions \cite[Theorem~6.3.8]{LawlerMR2677157}), we have that
\[\sup\limits_{y\in B(0,n)}\mathbb{P}_{(0),(2)}^{y,x_0}[X^0[0,n-\tau_1]\cap X^2[0,\infty)\neq\emptyset]=(1+O(K^{-1}))\cdot\mathbb{P}_{(0),(2)}^{0,x_0}[X^0[0,n-\tau_1]\cap X^2[0,\infty)\neq\emptyset].\]
Hence,
\begin{equation}\label{eq: key1}
 \mathbb{P}^{0,x_0,x_0}[\tau_1\leq\tau_2\leq n]\leq (1+O(K^{-1}))\mathbb{E}_{(0),(1)}^{0,x_0}\left[\begin{array}{l}X^0[0,n]\cap X^1[0,\infty)\neq\emptyset,\\ \mathbb{P}_{(0),(2)}^{0,x_0}[X^0[0,n-\tau_1]\cap X^2[0,\infty)\neq\emptyset]\end{array}\right].
\end{equation}
We denote by $k(n)$ the averaged capacity $\mathbb{E}[\Capa(X^0[0,n])]$. Then,
\begin{align}
 \mathbb{E}[\Capa(X^0[0,n])^2]\leq &(1+O(K^{-1}))s(4)K^{2}n^{2}\mathbb{E}^{0,x_0}_{(0),(1)}[2k(n-\tau_1),\tau_1\leq n]\notag\\
 = & (1+O(K^{-1}))s(4)K^{2}n^{2}\mathbb{E}^{0,x_0}_{(0),(1)}[2k(n-\tau_1)|\tau_1\leq n]\mathbb{P}^{0,x_0}_{(0),(1)}[\tau_1\leq n]\notag\\
 = & (1+O(K^{-1}))k(n)\mathbb{E}^{0,x_0}_{(0),(1)}[2k(n-\tau_1)|\tau_1\leq n].
\end{align}
Since $\mathbb{P}[\tau_1\leq s]=\mathbb{P}^{0,x_0}_{(0),(1)}[X^{0}[0,s]\cap X^{1}[0,\infty)\neq\emptyset]=(1+O(K^{-1}))(s(4))^{-1}K^{-2}n^{-2}\mathbb{E}[\Capa(X^0[0,s])]$, we have that
\[\mathbb{P}[\tau_1\leq s|\tau_1\leq n]=(1+O(K^{-1}))k(s)/k(n).\]
Hence, by \eqref{eq: asymp capacity 4d}, we see that
\begin{equation}\label{eq: cvg in dist first hitting time}
 \mathbb{P}[\tau_1/n\in\cdot|\tau_1\leq n]\rightarrow \text{Uniform distribution on }[0,1]\text{, as }K,n\rightarrow\infty.
\end{equation}
By \eqref{eq: asymp capacity 4d} and \eqref{eq: cvg in dist first hitting time}, we have that
\[\lim_{K,n\rightarrow\infty}\mathbb{E}^{0,x_0}_{(0),(1)}[2k(n-\tau_1)|\tau_1<n]/k(n)=1\]
and consequently,
\[\lim_{n\rightarrow\infty}\Var(X^0[0,n])/k(n)^2=0.\]
\end{proof}

\section{\texorpdfstring{Three dimension: $(\Capa(X[0,n])/\sqrt{n})_n$ has a random limit}{Three dimension: Cap(X[0,n])/√n has a random limit}}
In this section, for $d=3$, we will show that $\Capa(X[0,n])/\sqrt{n}$ converges to the corresponding quantity of a Brownian motion. The reason is that two $3D$ Brownian motion paths (or simple random walk paths) are very likely to intersect when they get close, see \cite[Lemmas~2.4 and 2.6]{LawlerMR1423466}. As an application of Skorokhod embedding theorem, one could closely couple SRWs and Brownian motions. For these two reasons, with a high probability, a SRW path is as hittable as a Brownian motion path. Accordingly, by last passage time decomposition and the scaling invariance of Brownian motion, we prove Theorem~\ref{thm: scaling limit capacity 3d}, which affirms the conjecture in \cite[Section~6]{SRWCapCLT}, see the remark below.
\begin{rem}\label{rem: 3d no weak LLN}
 It was conjectured that $\Capa(X[0,n])/\mathbb{E}[\Capa(X[0,n])]$ has a random limit as $n\rightarrow\infty$ for $d=3$, see \cite[Section~6]{SRWCapCLT}. By Theorem~\ref{thm: scaling limit capacity 3d}, we confirm this conjecture. Indeed, by considering a ball containing $X[0,n]$, it was proved that there exists $C<\infty$ such that $\mathbb{E}[(\Capa(X[0,n]))^2]\leq Cn$, see the proof of \cite[Lemma~5]{RathSapozhnikovMR2819660}. Hence, $(\Capa(X[0,n])/\sqrt{n})_n$ is uniformly integrable and $\lim_{n\rightarrow\infty}\mathbb{E}[\Capa(X[0,n])]/\sqrt{n}=\mathbb{E}[\frac{1}{3\sqrt{3}}\Capa_{BM}(M[0,1])]$, which implies that $\Capa(X[0,n])/\mathbb{E}[\Capa(X[0,n])]$ converges in distribution towards $\Capa_{BM}(M[0,1])/\mathbb{E}[\Capa_{BM}(M[0,1])]$ as $n\rightarrow\infty$. Note that $\Var(\Capa_{BM}(M[0,1]))>0$ since $\mathbb{E}[\Capa_{BM}(M[0,1])]>0$ and $\mathbb{P}[|M_t|\leq\epsilon,\forall t\in[0,1]]>0$ for all $\epsilon>0$, which, by monotonicity of capacities, implies that $\mathbb{P}[\Capa_{BM}(M[0,1])<\epsilon]>0$, $\forall\epsilon>0$.
\end{rem}

\begin{proof}[Proof of Theorem~\ref{thm: scaling limit capacity 3d}]
 By Skorokhod embedding, there exists a coupling between a pair of independent SRWs and a pair of independent Brownian motions. To be more precise, there exists a probability space such that the following holds, see \cite[Lemma~3.1]{LawlerMR1423466}.
 \begin{itemize}
  \item $(X^{0}_n)_n$ and $(X^{1}_n)_{n}$ are both SRWs on $\mathbb{Z}^3$ starting from $0$.
  \item $(M^{0}_t)_t$ and $(M^{1}_t)_{t}$ are both Brownian motions on $\mathbb{Z}^3$ starting from $0$.
  \item $(X^{0},M^{0})$ is independent of $(X^{1},M^{1})$.
  \item For all $\epsilon>0$, there exist $\gamma>0$ and $C<\infty$ such that for all $n\geq 1$,
        \begin{equation}\label{eq: Skorokhod approximation}
         \mathbb{P}\left[\sum_{i}\max_{s\leq n}|X^{i}_{\lfloor 3s\rfloor}-M^{i}_s|_{1}>n^{1/4+\epsilon}\right]\leq C\cdot e^{-n^{\gamma}}.
        \end{equation}
 \end{itemize}
 We take $\epsilon\leq\frac{1}{1000}$. We define several events as follows:
 \begin{itemize}
  \item $E_1\overset{\text{def}}{=}\{\max_{s\leq n}|X_{\lfloor 3s\rfloor}^{0}-M_s^{0}|_{1}\leq n^{\frac{1}{4}+\epsilon}\}$.
  \item $E_2\overset{\text{def}}{=}\{X^{0}[0,n]\subset B_{\mathbb{Z}^3}(0,n^{\frac{1}{2}+\frac{\epsilon}{4}})\}$.
  \item $E_{3,SRW}\overset{\text{def}}{=}\left\{\sup_{z\in \Nbd(X^{0}[0,n],n^{\frac{1}{2}-\epsilon})\cap\mathbb{Z}^3}\mathbb{P}\left[X^{0}[0,n]\cap (z+X^{1}[0,\infty))=\emptyset|X^{0}[0,n]\right]< n^{-\delta}\right\}$,
  where
  \[\Nbd(A,s)=\cup_{x\in A}B_{\mathbb{R}^3}(x,s),\quad\text{for }A\subset\mathbb{R}^3\text{ and }s\geq 0.\]
  Similarly, we define
  \[E_{3,BM}\overset{\text{def}}{=}\left\{\sup_{z\in \Nbd(M^{0}[0,n/3],n^{\frac{1}{2}-\epsilon})}\mathbb{P}\left[M^{0}[0,n/3]\cap (z+M^{1}[0,\infty))=\emptyset|M^{0}[0,n/3]\right]< n^{-\delta}\right\}.\]
  We define $E_3=E_{3,SRW}\cup E_{3,BM}$.
 \end{itemize}
 As we mentioned above, by Skorokhod approximation, $\mathbb{P}[E_1^c]\leq C\cdot e^{-n^{\gamma(\epsilon)}}$ where $C<\infty$ does not depend on $n$. By union bounds and Hoeffding's inequality, $\mathbb{P}[E_2^c]\leq Cne^{-n^{\epsilon/2}/C}$ where $C<\infty$ does not depend on $n$. By \cite[Lemmas~2.4, 2.6]{LawlerMR1423466}, for all $N\geq 1$, $\exists\delta>0$ (in the definition of $E_3$) and $C<\infty$ such that for all $n\geq 1$, $\mathbb{P}[E_3^c]\leq C\cdot n^{-N}$. We define $E=E_1\cap E_2\cap E_3$, take $N=2$ and choose $\delta$ accordingly such that
 \begin{equation}\label{eq: proba Ec}
  \exists C<\infty,\mathbb{P}[E^c]\leq C\cdot n^{-N}.
 \end{equation}
 Take $y_n\in\mathbb{Z}^d$ such that $||y_n||_2=\lfloor n^{\frac{1}{2}+\epsilon}\rfloor$. By the independence between $(X^0,M^0)$ and $(X^1,M^1)$, the last passage time decomposition and the Green function estimate, we have that
 \begin{equation}\label{eq: srw capacity and hitting proba}
  \Capa(X^{0}[0,n])1_{E}=\frac{2\pi}{3}(1+o(1))1_{E}n^{\frac{1}{2}+\epsilon}\mathbb{P}\left[X^{0}[0,n]\cap(y_n+X^{1}[0,\infty))\neq\emptyset|X^{0}[0,n],M^{0}[0,n/3]\right],
 \end{equation}
 and similarly, by \cite[Theorem~8.8, Theorem~8.27 and Definition~8.18]{MortersPeresMR2604525},
 \begin{equation}\label{eq: bm capacity and hitting proba}
  \Capa(M^{0}[0,n/3])1_{E}=2\pi(1+o(1))1_{E}n^{\frac{1}{2}+\epsilon}\mathbb{P}\left[M^{0}[0,n/3]\cap(y_n+M^{1}[0,\infty))\neq\emptyset|X^{0}[0,n],M^{0}[0,n/3]\right].
 \end{equation}
 We will show that $\eqref{eq: srw capacity and hitting proba}=\frac{1}{3}(1+o(1))\cdot\eqref{eq: bm capacity and hitting proba}$ which is equivalent to
 \begin{multline}\label{eq: hitting equivalence}
  1_{E}\mathbb{P}\left[X^{0}[0,n]\cap(y_n+X^{1}[0,\infty))\neq\emptyset|X^{0}[0,n],M^{0}[0,n/3]\right]\\
  =(1+o(1))1_{E}\mathbb{P}\left[M^{0}[0,n/3]\cap(y_n+M^{1}[0,\infty))\neq\emptyset|X^{0}[0,n],M^{0}[0,n/3]\right].
 \end{multline}
 We first find several quantities, which are asymptotically equivalent to the left hand side of \eqref{eq: hitting equivalence}. By the definition of $E_3$ and the strong Markov property of $X^{1}$, we get that
 \begin{multline}\label{eq: SRW path and tube plus SRW}
  1_{E}\mathbb{P}\left[X^{0}[0,n]\cap(y_n+X^{1}[0,\infty))\neq\emptyset|X^{0}[0,n],M^{0}[0,n/3]\right]\\
  =(1+o(1))1_{E}\mathbb{P}\left[\Nbd(X^{0}[0,n],n^{\frac{1}{2}-\epsilon})\cap(y_n+X^{1}[0,\infty))\neq\emptyset|X^{0}[0,n],M^{0}[0,n/3]\right]\\
  =(1+o(1))1_{E}\mathbb{P}\left[\Nbd(X^{0}[0,n],n^{\frac{1}{2}-4\epsilon})\cap(y_n+X^{1}[0,\infty))\neq\emptyset|X^{0}[0,n],M^{0}[0,n/3]\right].
 \end{multline}
 Next, we will show that $X^1$ could be replaced by $M^{1}$ in the following sense:
 \begin{equation}\label{eq: SRW path and tube and tube plus BM}
  1_{E}\mathbb{P}\left[\Nbd(X^{0}[0,n],n^{\frac{1}{2}-2\epsilon})\cap(y_n+M^{1}[0,\infty))\neq\emptyset|X^{0}[0,n],M^{0}[0,n/3]\right]=(1+o(1))\cdot\eqref{eq: SRW path and tube plus SRW},
 \end{equation}
 which would follow from Skorokhod approximation up to time $n^{1+8\epsilon}$ and the following three equations:
 \begin{multline}\label{eq: SRW path hit small SRW tube hitting time bound}
  1_{E}\mathbb{P}\left[\Nbd(X^{0}[0,n],n^{\frac{1}{2}-4\epsilon})\cap(y_n+X^{1}[0,\infty))\neq\emptyset|X^{0}[0,n],M^{0}[0,n/3]\right]\\
  =(1+o(1))1_{E}\mathbb{P}\left[\Nbd(X^{0}[0,n],n^{\frac{1}{2}-4\epsilon})\cap(y_n+X^{1}[0,n^{1+8\epsilon}))\neq\emptyset|X^{0}[0,n],M^{0}[0,n/3]\right],
 \end{multline}
 \begin{multline}\label{eq: SRW path hit big SRW tube hitting time bound}
  1_{E}\mathbb{P}\left[\Nbd(X^{0}[0,n],n^{\frac{1}{2}-\epsilon})\cap(y_n+X^{1}[0,\infty))\neq\emptyset|X^{0}[0,n],M^{0}[0,n/3]\right]\\
  =(1+o(1))1_{E}\mathbb{P}\left[\Nbd(X^{0}[0,n],n^{\frac{1}{2}-\epsilon})\cap(y_n+X^{1}[0,n^{1+8\epsilon}))\neq\emptyset|X^{0}[0,n],M^{0}[0,n/3]\right],
 \end{multline}
 and that
 \begin{multline}\label{eq: BM path hit SRW tube hitting time bound}
  1_{E}\mathbb{P}\left[\Nbd(X^{0}[0,n],n^{\frac{1}{2}-2\epsilon})\cap(y_n+M^{1}[0,\infty))\neq\emptyset|X^{0}[0,n],M^{0}[0,n/3]\right]\\
  =(1+o(1))1_{E}\mathbb{P}\left[\Nbd(X^{0}[0,n],n^{\frac{1}{2}-2\epsilon})\cap(y_n+M^{1}[0,n^{1+8\epsilon}/3))\neq\emptyset|X^{0}[0,n],M^{0}[0,n/3]\right]
 \end{multline}
 Indeed, by union bounds, Markov property, the last passage time decomposition and the Green function estimate, there exists $c>0$ such that
 \begin{multline}\label{eq: sphtub}
  1_{E}\mathbb{P}\left[\Nbd(X^{0}[0,n],n^{\frac{1}{2}-4\epsilon})\cap(y_n+X^{1}[n^{1+8\epsilon},\infty))\neq\emptyset|X^{0}[0,n],M^{0}[0,n/3]\right]\\
  \leq 1_{E}\mathbb{P}[|y_n+X^1_{n^{1+8\epsilon}}|\leq n^{\frac{1}{2}+2\epsilon}]+1_{E}\cdot c\cdot n^{-\frac{1}{2}-2\epsilon}\Capa(\Nbd(X^{0}[0,n],n^{\frac{1}{2}-4\epsilon}))\\
  \leq 1_{E}\cdot c\cdot (n^{-6\epsilon}+n^{-\frac{1}{2}-2\epsilon}\Capa(\Nbd(X^{0}[0,n],n^{\frac{1}{2}-4\epsilon}))).
 \end{multline}
 By the last passage time decomposition, the Green function estimate, monotonicity and translation invariance of the capacity and the estimate for the capacity of a ball, there exists $c>0$ such that
 \begin{multline}\label{eq: sphtlb}
  1_{E}\mathbb{P}\left[\Nbd(X^{0}[0,n],n^{\frac{1}{2}-4\epsilon})\cap(y_n+X^{1}[0,\infty))\neq\emptyset|X^{0}[0,n],M^{0}[0,n/3]\right]\\
  \geq 1_{E}\cdot c\cdot n^{-\frac{1}{2}-\epsilon}\Capa(\Nbd(X^{0}[0,n],n^{\frac{1}{2}-4\epsilon}))\\
  \geq 1_{E}\cdot c\cdot n^{-\frac{1}{2}-\epsilon}\Capa(B(0,n^{\frac{1}{2}-4\epsilon}))\geq 1_{E}\cdot c^2\cdot n^{-5\epsilon}.
 \end{multline}
 Comparing \eqref{eq: sphtub} with \eqref{eq: sphtlb}, we see that \eqref{eq: SRW path hit small SRW tube hitting time bound} holds. And \eqref{eq: SRW path hit big SRW tube hitting time bound} and \eqref{eq: BM path hit SRW tube hitting time bound} could be derived in a similar way. Next, we derive several quantities which are equivalent to the right hand side of \eqref{eq: hitting equivalence}. Similarly to \eqref{eq: SRW path and tube plus SRW}, we obtain that
 \begin{multline}\label{eq: BM path and big tube}
  1_{E}\mathbb{P}\left[\Nbd(M^{0}[0,n/3],n^{\frac{1}{2}-\epsilon})\cap(y_n+M^{1}[0,\infty))\neq\emptyset|X^{0}[0,n],M^{0}[0,n/3]\right]\\
  =(1+o(1))1_{E}\mathbb{P}\left[M^{0}[0,n/3]\cap(y_n+M^{1}[0,\infty))\neq\emptyset|X^{0}[0,n],M^{0}[0,n]\right].
 \end{multline}
 and
 \begin{multline}\label{eq: BM path and small tube}
  1_{E}\mathbb{P}\left[\Nbd(M^{0}[0,n/3],n^{\frac{1}{2}-4\epsilon})\cap(y_n+M^{1}[0,\infty))\neq\emptyset|X^{0}[0,n],M^{0}[0,n/3]\right]\\
  =(1+o(1))1_{E}\mathbb{P}\left[M^{0}[0,n/3]\cap(y_n+M^{1}[0,\infty))\neq\emptyset|X^{0}[0,n],M^{0}[0,n]\right].
 \end{multline}
 By the definition of $E_1$ (the Skorokhod approximation), for $n$ sufficient large, we have that
 \[\Nbd(M^{0}[0,n/3],n^{\frac{1}{2}-4\epsilon})\subset\Nbd(X^{0}[0,n],n^{\frac{1}{2}-2\epsilon})\subset\Nbd(M^{0}[0,n/3],n^{\frac{1}{2}-\epsilon})\]
 and hence, $\eqref{eq: BM path and small tube}\leq \eqref{eq: SRW path and tube and tube plus BM}\leq \eqref{eq: BM path and big tube}$. Therefore, \eqref{eq: hitting equivalence} holds and equivalently,
 \begin{equation}
  \Capa(X^{0}[0,n])1_{E}=\frac{1}{3}(1+o(1))\Capa(M^{0}[0,n/3])1_{E}.
 \end{equation}
 By Brownian scaling, $\Capa(M^{0}[0,n/3])$ has the same distribution as $\frac{\sqrt{n}}{\sqrt{3}}\Capa(M^{0}[0,1])$. Hence, together with \eqref{eq: proba Ec}, we get that $\Capa(X^{0}[0,n])/\sqrt{n}$ converges in distribution towards $\frac{1}{3\sqrt{3}}\Capa(M^{0}[0,1])$ as $n\rightarrow\infty$.
\end{proof}

\paragraph{Acknowledgement}
The author thanks Sapozhnikov for valuable advices and useful comments.

\providecommand{\bysame}{\leavevmode\hbox to3em{\hrulefill}\thinspace}
\providecommand{\MR}{\relax\ifhmode\unskip\space\fi MR }
\providecommand{\MRhref}[2]{%
  \href{http://www.ams.org/mathscinet-getitem?mr=#1}{#2}
}
\providecommand{\href}[2]{#2}

\end{document}